\documentclass[preprint]{elsarticle}
\usepackage{amsbsy}
\usepackage{amsfonts}
\usepackage{amsmath}
\usepackage{amssymb}

\newdefinition{defin}{Definition}
\newdefinition{ex}{Example}
\newdefinition{rem}{Remark}
\newtheorem{thm}{Theorem}
\newtheorem{prop}[thm]{Proposition}
\newproof{proof}{Proof}

\newtheorem{lem}{Lemma}

\begin{document}
\begin{frontmatter}
\title{The concept of self-similar automata over a changing alphabet and lamplighter groups generated by such automata}
\author[rvt]{Adam Woryna}
\ead{adam.woryna@polsl.pl}
\address[rvt]{Silesian University of Technology, Institute of Mathematics, ul. Kaszubska 23, 44-100 Gliwice, Poland}
\begin{abstract}
Generalizing the idea of self-similar groups defined by Mealy automata, we itroduce the notion of a self-similar automaton and a self-similar group over a changing alphabet. We show that every finitely generated residually-finite group is self-similar over an arbitrary unbounded changing alphabet. We construct some naturally defined self-similar automaton representations over an unbounded changing alphabet for any lamplighter group $K\wr \mathbb{Z}$ with an arbitrary finitely generated (finite or infinite) abelian group $K$.
\end{abstract}
\begin{keyword}
rooted tree  \sep changing alphabet \sep time-varying automaton \sep group generated by an automaton \sep lamplighter group
\end{keyword}
\end{frontmatter}

\section{Introduction and  main results}
The handling of various languages and approaches for describing groups of automorphisms of spherically homogeneous rooted trees brought  many significant results. For example, the algebraic language of the so-called 'tableaux' with  'truncated' polynomials over finite fields,   introduced by L.~Kaloujnine~\cite{17}, turned out to be effective in studying  these groups as iterated wreath products. V.~I.~Sushchanskyy~\cite{16} uses this language to construct the pioneering examples of  infinite periodic 2-generated $p$-groups ($p>2$) as well as to produce factorable subgroups of wreath products of groups~\cite{18,19}. A.~V.~Rozhkov~\cite{44}, by using the  notion of  the cortage and studying the so-called Aleshin type groups, first gave an example of a 2-generated periodic group containing elements of all possible finite orders (see also~\cite{66,55}). A.~Ershler~\cite{31,32} uses a combinatorial language of a time-varying automaton to discover new results concerning the growth and the so-called F${\rm \ddot{o}}$lner functions of groups.

In the present paper, by introducing the notion of a  self-similar automaton over a changing alphabet and the group generated by such an automaton, we  make an attempt to show that the concept of a self-similar group with its faithful action by automorphisms on  a homogeneous rooted tree (also called a regular rooted tree) can be naturally extended to an arbitrary spherically homogeneous rooted tree. As we show in Theorem~\ref{t1}, our combinatorial approach allows to characterize the class of finitely generated residually finite groups. We also provide (Theorem~\ref{t2})  naturally-defined representations for a large class of lamplighter groups for which the realization by the standard notion of a Mealy type automaton is not known.

\subsection{Mealy automata and self-similar groups}
Self-similar groups appear in many branches of mathematics, including operators algebra, dynamical systems, automata theory, combinatorics, ergodic theory, fractals,  and others. The most  interesting and beautiful examples of  self-similar actions are defined on the tree $X^*$ of finite words over a finite alphabet $X$ by using the notion of a transducer (also called a Mealy automaton), which, by definition, is a finite set $Q$ together with the transition function $\varphi\colon Q\times X\to Q$ and an output function $\psi\colon Q\times X\to X$. The convenient way defining and presenting a Mealy automaton is to draw its diagram. For example, the diagram in Figure~\ref{f1}
determines a Mealy automaton  over  the alphabet $X=\{0,1\}$, with five states $a, b, c, d, e$, with a transition function  given by the oriented edges and with an output function   given by labelling of the vertices by the maps $X\to X$, where $id$ denotes the identity map and $\sigma$ is a transposition: $0\mapsto 1, 1\mapsto 0$.
\begin{figure}[hbtp]
\centering
\includegraphics[width=4.5cm]{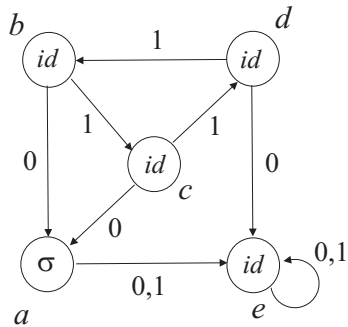}
\caption{The automaton generating the Grigorchuk group}
\label{f1}
\end{figure}
We easily verify from this diagram  the following equalities defining the transition and output functions of this automaton:
\begin{eqnarray*}
\varphi(a,x)&=&\varphi(e,x)=e,\;\;\;x\in X,\\
\varphi(b,0)&=&a,\;\; \varphi(b,1)=c,\;\;\varphi(c,0)=a,\\
\varphi(c,1)&=&d,\;\varphi(d,0)=e,\;\;\varphi(d,1)=b,\\
\psi(a,x)&=&\sigma(x),\;\psi(b,x)=\psi(c,x)=\psi(d,x)=\psi(e,x)=x,\;\;\;x\in X.
\end{eqnarray*}

Each state $q\in Q$ of a Mealy automaton $A=(X, Q, \varphi, \psi)$ defines a transformation of the tree $X^*$. Namely, given a word $x_0x_1\ldots x_k\in X^*$ the automaton $A$ acts on the first letter $x_0$ by the map which  labels the vertex corresponding to the state $q$ in the diagram of $A$, changes its state to the state given by the end of the arrow going out of $q$ and labelled by $x_0$; being in a new state the automaton reads the next letter $x_1$, transforms it and moves to the next state according to the above rule; and continues in this way to transform the remaining letters of the word. The induced  transformation of the set $X^*$  preserves the empty word and the beginnings of words and, if  $A$ is invertible (i.e. all the maps labelling  vertices in the diagram of $A$ are permutations of the alphabet $X$), it defines an automorphism of the tree $X^*$.

The group generated by automorphisms corresponding to all states of an invertible automaton $A$ with the composition of mappings as a product is called the group generated by this automaton. According to~\cite{0} (Definition~1.1, p. 129) an abstract group $G$ is called self-similar over an alphabet $X$ if there is a Mealy automaton $A=(X, Q, \varphi, \psi)$ which generates a group isomorphic with $G$.
For instance, the group $G$ generated by the automaton depicted in Figure~\ref{f1} is a fameous Grigorchuk group constituting one of the most interesting example among groups generated by all the states of a Mealy automaton and,  as it turned out, in the whole group theory. This  is also the first nontrivial example of a self-similar group.
Groups generated by the states of a Mealy automaton form a remarkable class of groups containing various important examples connected to many interesting topics. For more about Mealy automata, groups generated by them, self-similar groups and for many open problems around these concepts see the survey paper~\cite{1} or~\cite{4}.

\subsection{Time-varying automata and self-similarity over a changing alphabet}

A time-varying automaton, which is a natural generalization of a Mealy type automaton, allows to work not only with a fixed finite alphabet $X$ but also with an infinite sequence
$$
X=(X_i)_{i\geq 0}=(X_0, X_1, \ldots)
$$
of such  alphabets, which further will be called a  changing alphabet. We say that the changing alphabet $X$ is  unbounded if the sequence of cardinalities $(|X_i|)_{i\geq 0}$ (the so-called  branching index of $X$) is unbounded.

By definition, a time-varying automaton  over a changing alphabet $X=(X_i)_{i\geq 0}$ is a quadruple
$$
A=(X, Q, \varphi, \psi),
$$
where $Q$ is a finite set of states, $\varphi=(\varphi_i)_{i\geq 0}$ and $\psi=(\psi_i)_{i\geq 0}$ are sequences of, respectively,  transition functions and output functions of the form
$$
\varphi_i\colon Q\times X_i\to Q,\;\;\;\psi_i \colon Q\times X_i\to X_i.
$$
The automaton $A$ is called invertible if $x\mapsto \psi_i(q, x)$ is an invertible mapping of the set $X_i$ for all $i\geq 0$ and $q\in Q$. For every  $k\geq 0$ we also define the {\it $k$-th shift} of the automaton $A$ that is a time-varying automaton
$$
A_{k}=(X_{(k)}, Q, \varphi_{(k)}, \psi_{(k)}),
$$
where
$$
X_{(k)}=(X_{k+i})_{i\geq0},\;\;\;\varphi_{(k)}=(\varphi_{k+i})_{i\geq 0},\;\;\;\psi_{(k)}=(\psi_{k+i})_{i\geq 0}.
$$
In particular, $A$ is a Mealy automaton if and only if  $A_k=A$ for every $k\geq 0$.

At first, in the literature of  time-varying automata only the notion of an automaton over a fixed alphabet was considered and merely the structural properties of such automata were studied (see~\cite{9}). We first define in~\cite{11} the class of time-varying automata over a changing alphabet $X=(X_i)_{i\geq0}$ (see also~\cite{6,12,15}). Such an alphabet defines a tree $X^*$ of finite words, which is an example of a  spherically homogeneous rooted tree. The tree $X^*$ is  homogeneous if and only if the branching index of $X$ is constant. Moreover, any locally finite spherically homogeneous rooted tree $T$ is isomorphic with $X^*$ for a suitable changing alphabet $X=(X_i)_{i\geq 0}$.

We introduced in~\cite{11} the notion of a group generated by  an automaton over a changing alphabet and show   that this combinatorial language is apt to define and study groups of automorphisms of  spherically homogeneous rooted trees which may be not homogeneous.  Specifically, if $A=(X, Q, \varphi, \psi)$ is an invertible time-varying automaton over a changing alphabet $X=(X_i)_{i\geq 0}$, then for every $k\geq 0$ and each state $q\in Q$ we define, via transition and output functions, an automorphism  $q_k$ of the tree $X^*_{(k)}$  and call the group
$$
G(A_k)=\langle q_k\colon q\in Q\rangle
$$
as a {\it group generated by the $k$-th shift of $A$}.

\begin{defin}
An invertible time-varying automaton $A=(X, Q, \varphi, \psi)$ is called {\it self-similar} if for every $k\geq0$ the mapping $q_k\mapsto q_0$ ($q\in Q$) induces an isomorphism $G(A_{k})\simeq G(A)$.
\end{defin}

If  $A=(X, Q, \varphi, \psi)$ is a Mealy automaton, then we have $q_k=q_0$ for all $k\geq 0$ and $q\in Q$. Thus every Mealy automaton is self-similar by the above definition.

\begin{defin}
An abstract group $G$ is called a {\it self-similar group  over the changing alphabet $X=(X_i)_{i\geq 0}$} if there is a self-similar  automaton $A=(X, Q, \varphi, \psi)$ which generates a group isomorphic with $G$, i.e. $G(A)\simeq G$.
\end{defin}

Every self-similar group $G$ over a changing alphabet $X$ is an example of a finitely generated residually finite group (a group  is called residually finite if there is a descending chain $N_0\geq N_1\geq N_2\geq \ldots$ of normal subgroups of finite indexes and the trivial intersection). Indeed, the group $Aut(X^*)$ of all automorphisms of the tree $X^*$ is an example of a residually finite group (stabilizers of consecutive levels of this tree are normal subgroups of  finite indexes which intersect trivially). Thus $G$, as a group isomorphic with a subgroup of $Aut(X^*)$, is residually finite.

Given an abstract group $G$ and a changing alphabet $X=(X_i)_{i\geq 0}$, it is natural to ask about the {\it automaton complexity of the group $G$ over the alphabet $X$}, i.e. about the minimal number of states in an automaton  over  $X$ which generates a group isomorphic with $G$. Obviously, for any changing alphabet $X$ the automaton complexity of any group $G$ over $X$ is not smaller than the rank $r(G)$ of this group, i.e. than the minimal cardinality of a generating set of $G$.

\begin{defin}
If the number of states in an automaton $A=(X, Q, \varphi, \psi)$  coincides with the rank of the group $G(A)$, i.e. $|Q|=r(G(A))$, then the automaton $A$ is called {\it  optimal}.
\end{defin}

\subsection{The class of finitely generated self-similar groups over an unbounded changing alphabet}
In the theory of groups generated by Mealy automata the problem of classifying all  groups generated by  automata over a given alphabet $X$ and with a given number of states  is essential and intensively studied. For example, by the result from~\cite{1}  there are only six  groups (up to isomorphism) generated by 2-state Mealy automata over the binary alphabet, including the simplest example of an interesting group generated by a Mealy automaton, that is the lamplighter group $C_2\wr \mathbb{Z}$, where $C_2$ denotes a cyclic group of order 2 and $\mathbb{Z}$ is  the infinite cyclic group.
\begin{figure}[hbtp]
\centering
\includegraphics[width=6.3cm]{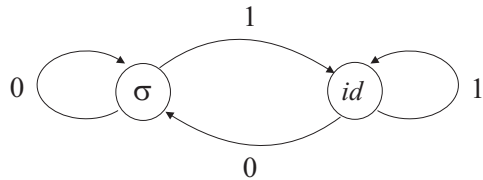}
\caption{A Mealy automaton generating the lamplighter group $C_2\wr \mathbb{Z}$}
\label{f3}
\end{figure}
Much effort put towards a classification of groups generated by 3-state Mealy automata over the binary alphabet. In the paper~\cite{7}  was shown, among other things, that there are no more than $124$ pairwise nonisomorphic groups defined by such automata.  Only recently it was discovered that a free-nonabelian group of finite rank is self-similar. Namely, M.~Vorobets and Ya.~Vorobets showed in~\cite{8} that a Mealy automaton (the so-called Aleshin type automaton) depicted in Figure~\ref{f2}  generates a free group of rank 3. Moreover, it is  still an open question whether or not there exists an optimal Mealy automaton which generates a free group of rank 2. Referring to this problem, we provided~\cite{15} for this group an explicit and naturally defined realization by an optimal automaton over a changing alphabet; see also Example~\ref{ex2}.
\begin{figure}[hbtp]
\centering
\includegraphics[width=5.5cm]{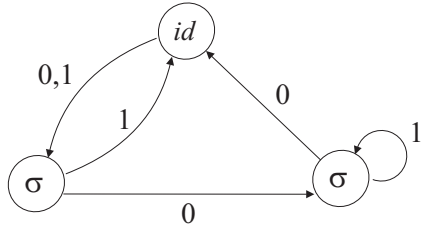}
\caption{A Mealy automaton generating a free group}
\label{f2}
\end{figure}

The first result of the present paper  shows that the situation completely changes if we try to characterize all self-similar groups over an arbitrary unbounded changing alphabet. Namely,  we show the following

\begin{thm}\label{t1}
Every finitely generated   residually finite group $G$  is self-similar over any unbounded changing alphabet and the corresponding self-similar automaton is optimal.
\end{thm}

As we will see in the proof, the construction leading to Theorem~\ref{t1} is based on the specific type of automaton which we call an automaton of a  diagonal type.

\begin{defin}
We say that a time-varying automaton $A=(X, Q, \varphi, \psi)$ is of {\it a diagonal type}, or that $A$ has transition functions defined in the {\it diagonal way} if $\varphi_i(q, x)=q$ for all $i\geq 0$, $q\in Q$ and $x\in X_i$.
\end{defin}

Note that the proof of Theorem~\ref{t1} is  unconstructive and we can not explicitly define the output functions of the corresponding self-similar automaton for a given group $G$ and a changing alphabet $X=(X_i)_{i\geq 0}$.   Indeed, in that proof we choose for the construction of an automaton $A$ which generates the group $G$  an arbitrary generating set $S$  as a set of states and we  define the transition functions in the diagonal way. Next, we show that the output functions of $A$ can be induced by some embedding $g\mapsto (g^{(0)}, g^{(1)}, \ldots)$ of the group $G$ into the infinite cartesian product $S(X_0)\times S(X_1)\times \ldots$ of symmetric groups on sets of letters; due to this embedding we define the output functions as follows: $\psi_i(s, x)=s^{(i)}(x)$ for all $i\geq 0$, $s\in S$ and $x\in X_i$. However, the above embedding arises simply from the fact that $G$ embeds into the cartesian product $G/N_0\times G/N_1\times\ldots$, where $G=N_0\triangleright N_1\triangleright N_2\triangleright \ldots$ is an infinite normal series with finite quotients $G/N_i$ ($i\geq 0$) and the trivial intersection. But the existence of such a normal series is equivalent to the assumption that $G$ is residually finite. Hence, given a residually finite group $G$, this method does not tell us how to compute effectively the permutations $s^{(i)}\in S(X_i)$. In particular, we know nothing about the cycle structure of these permutations. Moreover, in the case of a free group of rank~2 we tried  to find in the paper~\cite{6} an  explicit representation by an optimal  automaton of a diagonal type. However,  our computations led to a pretty complicated description of the mappings $s^{(i)}$ and, consequently, of output functions in such an automaton.

In view of  the above disadvantages as well as the previous results, it would be interesting  to find and study  explicit and naturally defined representations of a given group $G$ by self-similar automata over a changing alphabet, especially if  such representations are unknown by using the notion of a Mealy automaton.

\subsection{Lamplighter groups as self-similar groups over a changing alphabet}
Let $K$ be an arbitrary nontrivial finitely generated (finite or infinite) abelian group.  By definition, the lamplighter group $K\wr \mathbb{Z}$ is a semidirect product
$$
\bigoplus_\mathbb{Z}K\rtimes \mathbb{Z}
$$
with the action of the infinite cyclic group  on the direct sum $\bigoplus_\mathbb{Z}K$ by a shift. For the rank $r(K\wr\mathbb{Z})$ of this group we have: $r(K\wr \mathbb{Z})=r(K\times \mathbb{Z})=r(K)+1$. In recent years groups of this form are intensively studied via their self-similar actions. For example, the automaton realization of the lamplighter group $C_2\wr\mathbb{Z}$ from the paper~\cite{2} was used by R.~Grigorchuk and A.~${\rm \dot{Z}}$uk to compute the spectral measure associated on random walks on this group, which leads to a counterexample to the strong Atiyach conjecture~\cite{10}. This construction was generalized~\cite{3,5} by P.~Silva and B.~Steinberg in the concept of a reset Mealy automaton, where for every finite abelian group $K$ the corresponding lamplighter group  was realized as a group generated by a Mealy automaton in which the set of states and the alphabet coincides with $K$. Hence,  the only known optimal realization of the group $K\wr \mathbb{Z}$ by an automaton concerns the simplest case $K=C_2$. Here, it should also be mentioned that the question whether or not there exists a representation of a lamplighter group $K\wr \mathbb{Z}$ with an infinite  $K$  by a   Mealy automaton is still open.

For the second result of this paper we come from a naturally defined  and self-similar automaton $\mathcal{D}$ of a diagonal type over a changing alphabet $X$ that  generates  the direct product $K\times\mathbb{Z}$ (see Example~\ref{ex0}). We discover that only by a  simple manoeuvre  on each transition function in $\mathcal{D}$  we can pass to a new automaton $\mathcal{A}$ and obtain in this way a self-similar optimal automaton realization for the lamplighter group $K\wr \mathbb{Z}$. Namely,  we distinguish the state $a$ in the automaton $\mathcal{D}$  and  the letters $x_{q, i}\in X_i$ ($i\geq 0$) for each state  $q\neq a$ and we define  $\mathcal{A}$ as an automaton which operates like $\mathcal{D}$, unless being in a state $q\neq a$ it receives the letter $x_{q, i}$ in a moment $i\geq 0$, then it moves to the state~$a$.

To define the changing alphabet $X=(X_i)_{i\geq0}$ and output functions in the automaton $\mathcal{D}$ we decompose the group $K$ into a direct product of cyclic groups
$$
K=\mathbb{Z}^{n_1}\times (C_{r_1}\times\ldots\times C_{r_{n_2}}),\;\;\;n_1, n_2\geq 0
$$
such that $n_1+n_2=r(K)$. Let us denote $n:=r(K)=n_1+n_2$. For every $i\geq 0$ we consider arbitrary $2n$ pairwise disjoint  cycles
$$
\pi_{1,i}, \ldots, \pi_{n, i},\; \sigma_{1, i}, \ldots, \sigma_{n, i}
$$
in the symmetric group on a finite set $X_i$ which satisfy the following conditions: for each $1\leq s\leq n$ the cycles  $\pi_{s,i}$ ($i\geq 0$) do not have a uniformly bounded length, as well as the  cycles $\sigma_{s, i}$ ($i\geq 0$) for each $1\leq s\leq n_1$, and  the length of each cycle $\sigma_{s,i}$ ($n_1<s\leq n$, $i\geq 0$) does not depend on $i$ and   is equal to $r_{s-n_1}$. In particular, we see that the arising changing alphabet  $X=(X_i)_{i\geq0}$ is unbounded.

Now, taking $n$ symbols $b_1, \ldots, b_n$  different from $a$, we define the  automaton $\mathcal{D}$ as an automaton of a diagonal type over the changing alphabet $X=(X_i)_{i\geq0}$, with the set of states $Q=\{a, b_1,\ldots, b_{n}\}$ and the output functions  defined as follows:
$$
\psi_i(q,x)=\left\{\begin{array}{ll}\label{3}
\alpha_i(x),&q=a,\\
\beta_{s, i}(x), &q=b_s,\;1\leq s\leq n
\end{array}\right.
$$
for all $i\geq 0$, $q\in Q$ and $x\in X_i$, where the permutations $\alpha_i$, $\beta_{s,i}$  are defined in the following way
$$
\alpha_i=\pi_{1,i}\cdot\ldots\cdot\pi_{n, i},\;\;\;\;\;\beta_{s, i}= \sigma_{s, i}\cdot\alpha_i.
$$
As we mention above, passing to the  automaton $\mathcal{A}$ just needs to redefine  transition functions of $\mathcal{D}$. Explicitly, we define: $\mathcal{A}=(X, Q, \varphi^{\mathcal{A}}, \psi)$, where
$$
\varphi_i^{\mathcal{A}}(q, x)=\left\{\begin{array}{ll}\label{2}
a, &q=b_s,\;\;x=x_{s,i},\;\;1\leq s\leq n,\\
q, &\mbox{\rm otherwise},
\end{array}\right.
$$
and  $x_{s,i}:=x_{b_s,i}$ is a letter from the cycle $\pi_{s,i}$.

Our main result  is the following
\begin{thm}\label{t2}
The automaton $\mathcal{A}$ is self-similar and the group $G(\mathcal{A})$ is isomorphic with the lamplighter group $K\wr \mathbb{Z}$.
\end{thm}

In the proof of Theorem~\ref{t2} we do not use the algebraic language of  embedding into wreath products of groups (the so-called wreath recursion), which is  common in studying self-similar groups. Instead of this, for any $i\geq 0$ we directly provide  some  recursive formulae describing the action on  the tree $X^*_{(i)}$ of  automaton transformations  $a_i, c_{s,i}\colon X^*_{(i)}\to X^*_{(i)}$, where
$$
c_{s,i}=(b_s)_i\cdot a_i^{-1},\;\;\;1\leq s\leq n.
$$
We show (Proposition~\ref{prop1} and Lemma~\ref{lem1}) that for every $i\geq0$ any two transformations $c_{s, i}$ and $c_{s', i}$ ($1\leq s, s'\leq n$) commute and the mapping $c_{s, i}\mapsto \kappa_s$ induces an isomorphism between the group
$$
K_i=\langle c_{i,1}, \ldots, c_{i, n}\rangle
$$
and the group $K$, where the elements $\kappa_s\in K$ ($1\leq s\leq n$) form the standard generating set of $K$, that is $\kappa_s=(0, \ldots, 0, u_s, 0,\ldots, 0)$, where $u_s$ occurs in the $s$-th position and represents a generator of the corresponding cyclic factor of $K$. Next, for each $k\in\mathbb{Z}$ we consider the conjugation $K_i^{(k)}=a_i^{-k}K_ia_i^k$ as well as the group
$$
H_i=\langle K_i^{(k)}\colon k\in\mathbb{Z}\rangle
$$
generated by all these conjugations, and we show (Proposition~\ref{prop2}) that $H_i$ is a direct sum of its subgroups $K_i^{(k)}$, $k\in\mathbb{Z}$, which gives the isomorphism $H_i\simeq \bigoplus_\mathbb{Z}K$. Further, we show (Lemma~\ref{lem5}) that the group $H_i$ intersects trivially with the cyclic group generated by $a_i$ and that this cyclic group acts on $H_i$ by conjugation via the shift (Lemma~\ref{lem4}). This implies that the group
$$
G_i=\langle H_i, a_i\rangle=\langle c_{1, i}, \ldots, c_{n,i}, a_i\rangle
$$
is a semidirect product of the subgroups $H_i$ and $\langle a_i\rangle$, and, consequently, that $G_i$ is isomorphic with the lamplighter group $K\wr \mathbb{Z}$ via the mapping $a_i\mapsto u$, $c_{s, i}\mapsto \eta_s$, $1\leq s\leq n$, where $\{u, \eta_1,\ldots,\eta_n\}$ is the standard generating set of $K\wr \mathbb{Z}$, that is $u$ is a generator of $\mathbb{Z}$ in the  product $\bigoplus_\mathbb{Z}K\rtimes \mathbb{Z}$ and
$$
\eta_s=(\ldots,{\bf 0},{\bf 0},{\bf 0},\kappa_s,{\bf 0},{\bf 0},{\bf 0},\ldots)\in \bigoplus_\mathbb{Z}K,
$$
where each $\kappa_s$ on the right hand side occurs in zero position. Finally, we obtain that the group
$$
G(\mathcal{A}_{i})=\langle a_i, (b_1)_i, \ldots, (b_n)_i\rangle
$$
generated by the $i$-th shift of the automaton $\mathcal{A}$ is isomorphic with $G_i$ and the mapping $a_i\mapsto u$, $(b_s)_i\mapsto \eta_s\cdot u$ induces the isomorphism $G(\mathcal{A}_{i})\simeq K\wr \mathbb{Z}$.

As we will see in the proof, all calculations verifying the above propositions and lemmas are entirely elementary. However, they are burdened with some technical formulae and  the  presentation of the proof  seems to be quite formalistic. This  contrasts with a transparent and intuitive idea behind the definition of the automaton $\mathcal{A}$. On the other hand, our construction confirms one of the most intriguing phenomenon in the computational group theory which originates from dealing with the simplest examples of transducers, i.e. although  they  can generate various classic and important groups with rare and exotic properties, the derivation of  such a  group  from the corresponding automaton is  far from obvious and, quite often, requires  sophisticated and complex approaches. It can be seen, for example, in the original proof from~\cite{2} finding the lamplighter group $C_2\wr\mathbb{Z}$ as a self-similar group or in a long history discovering a free group  behind the Aleshin type automaton.

We hope that our  construction will not only fill the gap in self-similar automaton representations for a large class of lamplighter groups, but also  can serve as a useful tool in the further study of these groups, in particular, via their dynamics and geometry on the corresponding tree. We also think that it will help to better understand the connection between the structure of an automaton over a changing alphabet and the group it defines.

\section{The tree $X^*$ and its automaton transformations}
In the Introduction, we have already presented the notion of a changing alphabet $X=(X_i)_{i\geq 0}$, a time-varying automaton $A=(X, Q, \varphi, \psi)$, their shifts $X_{(k)}$, $A_k$ ($k\geq 0$)  and a self-similar group over a changing alphabet. Now, we  only recall the necessary notions and facts concerning the tree $X^*$ of finite words over $X$ and  automaton transformations $q_k\colon X^*_{(k)}\to X^*_{(k)}$ ($q\in Q$, $k\geq 0$).
We define the tree  $X^*$ as a disjoint sum of cartesian products
$$
X_0\times X_{1}\times\ldots\times X_{i},\;\;\;i\geq 0,
$$
together with the empty word $\epsilon$.  Since the elements of a word $w\in X^*$ are called letters, we will not separate them by commas and we will write $w=x_0x_1\ldots x_k$, where $x_i\in X_{i}$  are the letters. Thus for a given $i\geq 0$ the above cartesian product  consists of all words $w\in X^*$ with the length $|w|=i+1$. If $w\in X^*$ and $v\in X^*_{(|w|)}$, then $wv$ denotes a concatenation of  $w$ and $v$. Obviously $wv\in X^*$.
The set  $X^*$  has the structure of a spherically homogeneous rooted tree. The empty word is the root of this tree and two words are connected if and only if they are of the form $w$ and $wx$ for some $w\in X^*$ and  $x\in X_{|w|}$. In particular, the $i$-th level $X^i$ ($i\geq0$)  of the tree $X^*$ (that is the set of vertices at the distance $i$ from the root) consists of all words of the length $i$.

There is a natural interpretation of an automaton $A=(X, Q, \varphi, \psi)$ over a changing alphabet $X=(X_i)_{i\geq 0}$ as a machine, which being in a moment $i\geq 0$ in a state $q\in Q$ reads a letter $x\in X_i$ from the input tape, types the letter $\psi_i(q, x)$ on the output tape, moves to the state $\varphi_i(q, x)$ and proceeds further to the moment $i+1$. This  allows to define for every $k\geq0$ and  $q\in Q$     the transformation $q_k\colon X^*_{(k)}\to X^*_{(k)}$ recursively: $q_k(\epsilon)=\epsilon$ and if $xw\in X^*_{(k)}$ with $x\in X_k$, $w\in X^*_{(k+1)}$, then
\begin{equation}\label{0}
q_k(xw)=x'q'_{k+1}(w),
\end{equation}
where $x'=\psi_k(q,x)$ and $q'=\varphi_k(q, x)$. The transformation $q_k$ is called an automaton transformations of the tree $X^*_{(k)}$ (corresponding to the state $q$ in the $k$-th transition of $A$).

It is convenient to present an automaton $A=(X, Q, \varphi, \psi)$ over a changing alphabet $X=(X_i)_{i\geq 0}$ as a labeled directed locally finite graph with the  set
$$
\{(i, q)\colon i\geq0,\; q\in Q\}
$$
as a set of vertices. Two vertices are connected with an arrow if and only if they are of the form $(i, q)$ and $(i+1, \varphi_i(q, x))$ for some $i\geq0$, $q\in Q$ and $x\in X_i$. This arrow is labeled by $x$, starts from the vertex $(i, q)$ and goes to the vertex $(i+1, \varphi_i(q, x))$. Each vertex $(i, q)$ is labeled by the mapping
$$
\tau_{i, q}\colon X_i\to X_i,\;\;\;\tau_{i,q}\colon x\mapsto \psi_i(q,x)
$$
(the so-called labelling of the state $q$ in the $i$-th transition of $A$).   Further, to make the graph of the automaton clear,  we will substitute a large number of arrows connecting two given vertices and having  the same direction for a one multi-arrow labeled by suitable letters and if the labelling of such a multi-arrow starting from a given vertex follows from the labelling of other arrows starting from this vertex, we will omit this labelling.
\begin{figure}[hbtp]
\centering
\includegraphics[width=9cm]{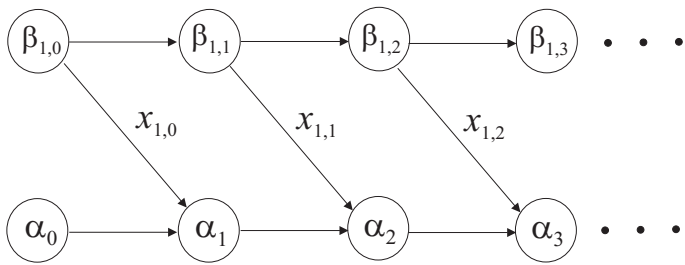}
\caption{The graph of the automaton $\mathcal{A}$ generating $K\wr \mathbb{Z}$, case $n=1$}
\label{fig14}
\end{figure}
Directly by the formula~(\ref{0}) we see that in case  $A$ invertible (i.e. each $\tau_{i, q}$ is a permutation of the set $X_i$), every automaton transformation $q_i\colon X^*_{(i)}\to X^*_{(i)}$ is also invertible. By the formula (\ref{0}) we also see that $q_i$  preserves the lengths of words and common beginnings and hence it defines an endomorphisms of the tree $X^*_{(i)}$. Moreover, the composition of two automaton transformations of the same tree is also an automaton transformation of this tree and the inverse mapping to the invertible automaton transformation is also of this type (we can show this by using an operation of the composition  of two time-varying automata over the same  changing alphabet and the notion of an inverse automaton - see~\cite{11}). In particular, the set of all invertible automaton transformations of a tree $X^*$ together with the composition of mappings as a product forms a proper subgroup of the group $Aut(X^*)$ of automorphisms of this tree, which is denoted by  $FA(X^*)$ and called the group of finite-state automorphisms.

\begin{rem}
Through the rest of the text we use the right action convention in writing the composition $fg$ of two transformation $f, g\colon X^*\to X^*$, i.e. the transformation applied first is written on the left and for any word $w\in X^*$ we have the equality $fg(w)=g(f(w))$.
\end{rem}

\section{The proof of Theorem~\ref{t1}}
Since the group $G$ is countable and residually finite, there is  an infinite normal series
$$
G=N_0\triangleright N_1\triangleright\ldots
$$
with finite quotients $G/N_i$ and the trivial intersection. Then the function
$g\mapsto (gN_i)_{i\geq 0}$
is an embedding of $G$ into the infinite cartesian product
\begin{equation}\label{cp}
G/N_0\times G/N_1\times\ldots.
\end{equation}
Let $X=(X_i)_{i\geq 0}$ be an unbounded changing alphabet. Then there is an infinite increasing sequence $0<t_0<t_1<\ldots$ of integers such that for every $i\geq 0$ the quotient $G/N_i$ embeds into the product
$$
S(X_{t_{i-1}+1})\times S(X_{t_{i-1}+2})\times\ldots\times S(X_{t_i})
$$
of  symmetric groups of  consecutive sets of letters, where for $i=0$ we take $t_{(-1)}=-1$. Consequently, there is an
embedding
$$
g\mapsto (g^{(0)}, g^{(1)}, \ldots),\;\;\;g\in G
$$
of  $G$ into the cartesian product $S(X_0)\times S(X_1)\times\ldots$.  The above embedding  allows to define the transformations $\xi_{g,i}\colon X^*_{(i)}\to X^*_{(i)}$ ($i\geq 0$, $g\in G$) as follows: $\xi_{g,i}(\epsilon)=\epsilon$ and if $w=x_0x_1\ldots x_k\in X^*_{(i)}$, then
$$
\xi_{g,i}(w)=g^{(i)}(x_0)g^{(i+1)}(x_1)\ldots g^{(i+k)}(x_k).
$$
In particular, for all $g, g'\in G$ and  $i\geq 0$ we have
\begin{equation}\label{epim}
\xi_{gg',i}=\xi_{g,i}\cdot \xi_{g',i},\;\;\;\xi_{g^{-1},i}=(\xi_{g,i})^{-1}.
\end{equation}

Let $S$ be any generating set of $G$ whose cardinality coincides with the rank of $G$. Let us define a time-varying automaton $A=(X, S, \varphi, \psi)$ as follows:
$$
\varphi_i(s, x)=s,\;\;\;\psi_i(s,x)=s^{(i)}(x)
$$
for all $i\geq 0$, $s\in S$, $x\in X_i$. Let us fix $i\geq 0$. Directly by the definition of $A$ the  transformation $s_i$ defined by a state $s\in S$ in the $i$-th transition of $A$ acts on any word $w=x_0x_1\ldots x_k\in X^*_{(i)}$ in the following way
$$
s_i(w)=s^{(i)}(x_0)s^{(i+1)}(x_1)\ldots s^{(i+k)}(x_k).
$$
In particular, we have $s_i=\xi_{s,i}$ for every $s\in S$ and, since the elements $s_i$ for  $s\in S$ generate the group $G(A_i)$, we have
$$
G(A_i)=\{\xi_{g,i}\colon g\in G\}.
$$
By the equalities~(\ref{epim}) we see that the mapping
$$
G\to G(A_i),\;\;\;g\mapsto \xi_{g,i}
$$
is an epimorphism. Further, since $\bigcap_{k\geq 0}N_k=\{1_G\}$, for every $g\neq 1_G$ there are infinitely many $k\geq 0$ such that $gN_k\neq N_k$. Consequently $g^{(k)}\neq Id_{X_k}$ for infinitely many  $k\geq 0$, which  implies $\xi_{g,i}\neq Id_{X^*_{(i)}}$ for every $g\neq 1_G$. Hence the above epimorphism is one-to-one. Consequently, the mapping $s\mapsto s_i$ ($s\in S$) induces an isomorphism $G\simeq G(A_i)$, i.e. the automaton $A$ is an optimal self-similar realization of $G$.

As we explained in the introduction, in spite of Theorem~\ref{t1}, it is  not easy to find an explicit self-similar automaton realization for a given group $G$. Even if we are able to construct an automaton defining $G$, this automaton may not be necessarily self-similar.

\begin{ex}\label{ex2}
Let $(r_i)_{i\geq 0}$ be a nondecreasing unbounded sequence of integers $r_i>1$ and let $X=(X_i)_{i\geq0}$ be a changing alphabet in which $X_i=\{1,2, \ldots, r_i\}$. For each $i\geq 0$   we define two cyclic permutations $\alpha_i$, $\beta_i$ of the set $X_i$: $\alpha_i=(1,2)$, $\beta_i=(1,2\ldots,r_i)$.
Let $A=(X, \{a, b\},  \varphi, \psi)$ be an automaton in which the transition and output functions are defined as follows
$$
\varphi_i(q, x)=\left\{
\begin{array}{ll}
q, &x\neq1,\\
a, &x=1,\;q=b,\\
b,& x=1,\; q=a,
\end{array}
\right.\;\;\;\psi_i(q, x)=\left\{
\begin{array}{ll}
\alpha_i(x), &q=a,\\
\beta_i(x), &q=b.
\end{array}
\right.
$$
\begin{figure}[hbtp]
\centering
\includegraphics[width=9.5cm]{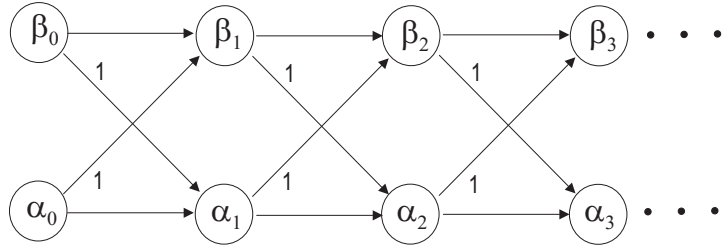}
\caption{A self-similar automaton from Example~\ref{ex2}}
\label{fig11}
\end{figure}
According to~\cite{15}  the group $G(A)=\langle a_0, b_0\rangle$ is  a free nonabelian group of rank two generated freely by the transformations $a_0$, $b_0$. Directly by the construction of the automaton  $A$ we see that for every $i\geq 0$ the  automaton transformations $a_i, b_i\colon X^*_{(i)}\to X^*_{(i)}$ also generate a free nonabelian group of rank 2. In particular, the mapping $a_i\mapsto a_0$, $b_i\mapsto b_0$ induces an isomorphism $G(A_{i})\simeq G(A)$. Thus the automaton $A$ is a self-similar realization of a free nonabelian group of rank 2 over the changing alphabet $X$.
\end{ex}

\begin{ex}\label{ex1}
Let the changing alphabet $X=(X_i)_{i\geq0}$ and the permutations $\alpha_i, \beta_i\in S(X_i)$ be defined as in the previous example. Let $A=(X, \{a, b\}, \varphi, \psi)$ be an automaton of a diagonal type with the output functions defined as follows: $\psi_i(a, x)=\alpha_i(x)$ and $\psi_i(b, x)=\beta_i(x)$ for all $i\geq 0$ and $x\in X_i$.
\begin{figure}[hbtp]
\centering
\includegraphics[width=9cm]{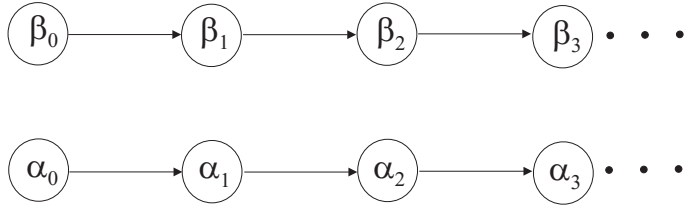}
\caption{An  automaton from Example~\ref{ex1}}
\label{fig12}
\end{figure}
The automaton $A$ is not a self-similar  realization of the group $G=G(A)$ because $A$ is not self-similar. Indeed, let $t\geq 0$ be the smallest number such that $r_{t}\geq 3$ and  $t'> t$ be the smallest number such that $r_{t'}>r_t$.
Then one  can verify by straightforward calculations that  $W(a_{t'}, b_{t'})\neq Id_{X^*_{(t')}}$ and $W(a_0, b_0)=Id_{X^*}$, where the group-word $W=W(a,b)$ is defined as follows $W(a,b)=(ab^{r_t-1}ab^{-r_t+1})^2$.
Thus the mapping $a_{t'}\mapsto a_{0}$, $b_{t'}\mapsto b_{0}$ can not be extended to the isomorphism of groups $G(A_{t'})$ and $G(A)$.
\end{ex}

\begin{ex}\label{ex0}
Let $\mathcal{D}$ be an automaton of a diagonal type defined in the introduction. For every $i\geq 0$ let $a_i$, $(b_s)_i$ ($1\leq s\leq n$) be automaton transformations of the tree $X^*_{(i)}$ corresponding to  the states of  $\mathcal{D}$ in  $i$-th transition. Obviously, the automaton $\mathcal{D}$ is invertible. By  definition of transition and output functions in $\mathcal{D}$ we obtain the following recursions: $a_i(xw)=\alpha_i(x)a_{i+1}(w)$, $(b_s)_i(xw)=\beta_{s,i}(x)b_{i+1}(w)$
for every $x\in X_i$ and $w\in X^*_{(i+1)}$. The group $G(\mathcal{D}_i)$ generated by the $i$-th shift of the automaton $\mathcal{D}$ is generated by the transformations $a_i$ and $c_{s,i}=(b_s)_ia_i^{-1}$ for $1\leq s\leq n$. Since $\sigma_{s,i}=\beta_{s,i}\alpha_i^{-1}=\alpha_i^{-1}\beta_{s,i}$, we have by the above recursions:
$c_{s,i}(xw)=\sigma_{s,i}(x)c_{s,i+1}(w)$ for  $x\in X_i$ and $w\in X^*_{(i+1)}$.
We consider in the infinite cartesian product $S(X_0)\times S(X_1)\times\ldots$  the subgroup $\langle \lambda_i, \gamma_{1,i}, \ldots, \gamma_{n,i}\rangle$,
where
$$
\lambda_i=(\alpha_i, \alpha_{i+1}, \ldots),\;\;\;\gamma_{s,i}=(\sigma_{s,i}, \sigma_{s, i+1},\ldots).
$$
By the above recursions for $a_i$ and $c_{s, i}$ we see that the mapping $a_i\mapsto \lambda_i$, $c_{s, i}\mapsto \gamma_{s,i}$, $1\leq s\leq n$ induces an isomorphism $G(\mathcal{D}_i)\simeq \langle \lambda_i, \gamma_{1,i}, \ldots, \gamma_{n,i}\rangle$. For each $i\geq 0$ the permutations $\alpha_i, \sigma_{1,i}, \ldots, \sigma_{n, i}$ are pairwise disjoint. Consequently,  the group $\langle \lambda_i, \gamma_{1,i}, \ldots, \gamma_{n,i}\rangle$ is abelian and the product $\lambda_i^{m_0}\gamma_{1,i}^{m_1}\ldots\gamma_{n,i}^{m_n}$ is trivial if and only if $o(\lambda_i)\mid m_0$ and $o(\gamma_{s,i})\mid m_s$ for $1\leq s\leq n$, where $o(g)$ denotes the order of an element $g$. Since $\alpha_i=\pi_{1,i}\cdot\ldots\cdot\pi_{n, i}$, we  see by  the assumptions on the permutations $\pi_{s,i}$ and $\sigma_{s,i}$ that $o(\lambda_i)=o(\gamma_{1,i})=\ldots=o(\gamma_{n_1,i})=\infty$ and $o(\gamma_{s,i})=r_{s-n_1}$ for $n_1<s\leq n$. This implies that the automaton $\mathcal{D}$ is self-similar and it generates the group isomorphic with $K\times \mathbb{Z}$.
\end{ex}

\section{The proof of Theorem~\ref{t2}}
Let $\mathcal{A}$ be an automaton defined in the introduction and for every $i\geq 0$ let $a_i$, $(b_s)_i$ ($1\leq s\leq n$) be automaton transformations of the tree $X^*_{(i)}$ corresponding to  the states of $\mathcal{A}$ in its $i$-th transition.
 Obviously, the automaton $\mathcal{A}$ is invertible. Directly by formulae defining transition and output functions of $\mathcal{A}$  we obtain the following recursions:
$$
a_i(xw)=\alpha_i(x)a_{i+1}(w),\;\;\;(b_s)_i(xw)=\left\{\begin{array}{ll}\beta_{s,i}(x)a_{i+1}(w),& x=x_{s,i},\\
\beta_{s, i}(x)(b_s)_{i+1}(w),& x\neq x_{s, i}\end{array}\right.
$$
for every $x\in X_i$ and $w\in X^*_{(i+1)}$. Hence $a_i^{-1}(xw)=\alpha_i^{-1}(x)a^{-1}_{i+1}(w)$ and the transformations $c_{s,i}=(b_s)_ia_i^{-1}$ satisfy:
\begin{eqnarray*}
c_{s, i}(xw)&=&(b_s)_ia_i^{-1}(xw)=a_i^{-1}((b_s)_i(xw))=\\
&=&\left\{
\begin{array}{ll}
a_i^{-1}(\beta_{s, i}(x)a_{i+1}(w)),&x=x_{s, i},\\
a_i^{-1}(\beta_{s, i}(x)(b_s)_{i+1}(w)), &x\neq x_{s, i},
\end{array}
\right.=\\
&=&\left\{
\begin{array}{ll}
\alpha_i^{-1}(\beta_{s, i}(x))a^{-1}_{i+1}(a_{i+1}(w)),&x=x_{s, i},\\
\alpha_i^{-1}(\beta_{s, i}(x))a^{-1}_{i+1}((b_s)_{i+1}(w)), &x\neq x_{s, i},
\end{array}
\right.=\\
&=&\left\{
\begin{array}{ll}
\sigma_{s, i}(x)w,&x=x_{s, i},\\
\sigma_{s, i}(x)c_{s, i+1}(w), &x\neq x_{s, i}.
\end{array}
\right.
\end{eqnarray*}
But, if $x=x_{s, i}$, then $\sigma_{s, i}(x)=x$ as the cycles $\pi_{s,i}$ and $\sigma_{s,i}$ are disjoint. In consequence, we have
\begin{equation}\label{5}
c_{s,i}(xw)=\left\{\begin{array}{ll}xw,& x=x_{s,i},\\
\sigma_{s, i}(x)c_{s, i+1}(w),& x\neq x_{s, i}.\end{array}\right.
\end{equation}

\subsection{The groups $K_i$}
We consider the groups
$$
K_i=\langle c_{1, i}, \ldots, c_{n, i}\rangle,\;\;\;i\geq 0.
$$

\begin{prop}\label{prop1}
For every $i\geq 0$ the mapping $c_{s, i}\mapsto \kappa_s$ ($1\leq s\leq n$) induces an isomorphism  $K_i\simeq K$.
\end{prop}
\begin{proof}
At first we show that for $s\neq s'$ the transformations $c_{s, i}$ and $c_{s', i}$ commute.  By  (\ref{5})    we obtain:
\begin{eqnarray}
c_{s, i}c_{s',i}(xw)&=&c_{s',i}(c_{s,i}(xw))=\nonumber\\
&=&\left\{\begin{array}{ll}
c_{s',i}(xw),&x=x_{s,i},\\
c_{s',i}(\sigma_{s,i}(x)c_{s,i+1}(w)), &x\neq x_{s, i}.\label{10}
\end{array}\right.
\end{eqnarray}
If $x=x_{s, i}$, then $x\neq x_{s',i}$ and $\sigma_{s',i}(x)=x$. Thus by (\ref{5}) and (\ref{10})  we have  in this case
$$
c_{s, i} c_{s',i}(xw)=c_{s', i}(xw)=\sigma_{s',i}(x)c_{s', i+1}(w)=xc_{s',i+1}(w).
$$
Similarly, if $x=x_{s',i}$, then  $\sigma_{s, i}(x)=x$ and by (\ref{5}) and (\ref{10}) we obtain in this case
$$
c_{s, i} c_{s',i}(xw)=c_{s',i}(\sigma_{s,i}(x)c_{s,i+1}(w))=c_{s',i}(xc_{s, i+1}(w))=xc_{s, i+1}(w).
$$
If $x\neq x_{s, i}$ and $x\neq x_{s',i}$, then $\sigma_{s, i}(x)\neq x_{s',i}$ and consequently
$$
c_{s, i} c_{s',i}(xw)=c_{s',i}(\sigma_{s,i}(x)c_{s,i+1}(w))=\sigma_{s, i}\sigma_{s',i}(x)c_{s, i+1}c_{s',i+1}(w).
$$
We conclude
\begin{equation}\label{14}
c_{s, i} c_{s',i}(xw)=\left\{\begin{array}{ll}
xc_{s',i+1}(w),&x=x_{s,i},\\
xc_{s,i+1}(w), &x=x_{s',i},\\
\sigma_{s,i}\sigma_{s',i}(x)c_{s,i+1}c_{s',i+1}(w), & \mbox{\rm otherwise}.
\end{array}\right.
\end{equation}
By analogy we have
$$
c_{s', i} c_{s,i}(xw)=\left\{\begin{array}{ll}
xc_{s',i+1}(w),&x=x_{s,i},\\
xc_{s,i+1}(w), &x=x_{s',i},\\
\sigma_{s',i} \sigma_{s,i}(x)c_{s',i+1} c_{s,i+1}(w), &\mbox{\rm otherwise}.
\end{array}\right.
$$
Thus we may write
\begin{equation}\label{eq111}
c_{s, i} c_{s',i}(xw)=\left\{
\begin{array}{ll}
c_{s',i} c_{s, i}(xw),& x\in\{x_{s, i}, x_{s',i}\},\\
\sigma_{s, i}\sigma_{s',i}(x)c_{s,i+1} c_{s',i+1}(w),&\mbox{\rm otherwise}.
\end{array}
\right.
\end{equation}
Since the cycles $\sigma_{s,i}$ and $\sigma_{s',i}$  commute, we see by~(\ref{eq111}) that  the transformations $c_{s, i}$ and $c_{s', i}$  commute as well.

If $x\neq x_{s,i}$, then $\sigma_{s,i}^{-1}(x)\neq x_{s,i}$ and by the recursion (\ref{5}) we have:
$$
c^{-1}_{s,i}(xw)=\left\{\begin{array}{ll}xw,& x=x_{s,i},\\
\sigma^{-1}_{s, i}(x)c^{-1}_{s, i+1}(w),& x\neq x_{s, i}\end{array}\right.
$$
By easy inductive argument we obtain for any integer $k$:
\begin{equation}\label{9}
c^k_{s,i}(xw)=\left\{\begin{array}{ll}xw,& x=x_{s,i},\\
\sigma_{s, i}^k(x)c^k_{s, i+1}(w),& x\neq x_{s, i}.\end{array}\right.
\end{equation}
By our assumption, for each $1\leq s\leq n_1$ the  cycles $\sigma_{s, i}$ ($i\geq 0$) do not have a uniformly bounded length. Hence, by  the  recursion (\ref{9}) we see  that the transformation $c_{s, i}$ is of infinite order for every $1\leq s\leq n_1$. Similarly, if $n_1<s\leq n$, then the length of  $\sigma_{s, i}$ is equal to $r_{s-n_1}$ and consequently the order of $c_{s, i}$ is equal to $r_{s-n_1}$.

We define the subset $\mathcal{V}\subseteq \mathbb{Z}^n$ as follows:
$$
\mathcal{V}=\{(m_1, \ldots, m_n)\colon 0\leq m_s<r_{s-n_1}\;\;{\rm for}\;\;n_1<s\leq n\}.
$$
By above, we obtain the equality
$$
K_i=\{C_{M, i}\colon M\in \mathcal{V}\},
$$
where for  every $M=(m_1,\ldots, m_n)\in\mathcal{V}$ we define
$$
C_{M,i}=c_{1, i}^{m_1}\cdot\ldots \cdot c_{n, i}^{m_n}.
$$
We show  that the above presentation is unique. To this end  to each $C_{M,i}$ we associate the permutation $\Sigma_{M, i}\in Sym(X_i)$ defined  as follows:
$$
\Sigma_{M, i}=\sigma_{1, i}^{m_1}\cdot\ldots\cdot\sigma_{n_,i}^{m_n}.
$$
In the following lemma, which is a generalization of the recursion (\ref{9}), for every $M\in\mathcal{V}$ we denote by $M^{(s)}$ ($1\leq s\leq n$)   the element arising from $M$ by the substitution of $m_s$ for $0$.
\begin{lem}\label{lem1}
We have
\begin{equation}\label{6}
C_{M, i}(xw)=\left\{
\begin{array}{ll}
xC_{M^{(s)}, i+1}(w), &x=x_{s,i},\\
\Sigma_{M, i}(x)C_{M, i+1}(w), &otherwise
\end{array}
\right.
\end{equation}
for all $x\in X_i$,  $w\in X^*_{(i+1)}$ and $M\in\mathcal{V}$.
\end{lem}
\begin{proof}(of Lemma~\ref{lem1})
Let $M=(m_1, \ldots, m_n)$ and for every $1\leq s\leq n$ let $M_s\in \mathcal{V}$  arises from $M$ by the substitution of every component $m_i$ with $i>s$ for 0. Since $M_n=M$ and for every $1\leq s\leq n$ we have
$$
C_{M_s, i}=c_{1, i}^{m_1}\cdot\ldots \cdot c_{s, i}^{m_s},
$$
it is sufficient to show that for every $1\leq s\leq n$ the following recursion holds:
\begin{equation}\label{11}
C_{M_s, i}(xw)=\left\{
\begin{array}{ll}
xC_{M^{(s')}_s, i+1}(w), &x=x_{s',i},\;1\leq s'\leq s,\\
\Sigma_{M_s, i}(x)C_{M_s, i+1}(w), &\mbox{\rm otherwise}.
\end{array}
\right.
\end{equation}
For $s=1$ we have $C_{M_s,i}=C_{M_1,i}=c_{1, i}^{m_1}$ and (\ref{11}) coincides with the recursion~(\ref{9}). Let us fix $s>1$ and let us assume that (\ref{11}) is true for  $s-1$. Then we have
\begin{eqnarray*}
C_{M_s, i}(xw)&=&c_{s, i}^{m_s}(C_{M_{s-1}, i}(xw))=\\
&=&\left\{\begin{array}{ll}
c_{s, i}^{m_s}(xC_{M_{s-1}^{(s')},i+1}(w)), & x=x_{s', i},\;\;1\leq s'<s,\\
c_{s, i}^{m_s}(\Sigma_{M_{s-1},i}(x)C_{M_{s-1}, i+1}(w)), &\mbox{\rm otherwise}.
\end{array}\right.
\end{eqnarray*}
But, if $x=x_{s',i}$ for some $1\leq s'<s$, then $\sigma_{s, i}^{m_s}(x)=x$. Thus by the recursion~(\ref{9}) we have in this case
$$
c_{s, i}^{m_s}(xC_{M_{s-1}^{(s')},i+1}(w))=\sigma_{s, i}^{m_s}(x)c_{s, i+1}^{m_s}(C_{M_{s-1}^{(s')},i+1}(w))=
xC_{M_s^{(s')}, i+1}(w).
$$
If $x=x_{s, i}$, then $\Sigma_{M_{s-1},i}(x)=\sigma^{m_1}_{1,i}\cdot\ldots\cdot \sigma^{m_{s-1}}_{s-1,i}(x)=x$. Thus by the recursion~(\ref{9}) we have in this case
\begin{eqnarray*}
C_{M_s, i}(xw)&=&c_{s, i}^{m_s}(\Sigma_{M_{s-1},i}(x)C_{M_{s-1}, i+1}(w))=\\
&=&c_{s, i}^{m_s}(xC_{M_{s-1},i+1}(w))=\\
&=&xC_{M_{s-1},i+1}(w)=\\
&=&xC_{M_s^{(s)},i+1}(w).
\end{eqnarray*}
If $x\notin\{x_{1 i},\ldots, x_{s, i}\}$, then $\Sigma_{M_{s-1}, i}(x)\neq \Sigma_{M_{s-1}, i}(x_{s, i})=x_{s,i}$. Thus by the recursion~(\ref{9}) we have in this case
\begin{eqnarray*}
C_{M_s, i}(xw)&=&c_{s, i}^{m_s}(\Sigma_{M_{s-1},i}(x)C_{M_{s-1}, i+1}(w))=\\
&=&\sigma_{s, i}^{m_s}(\Sigma_{M_{s-1},i}(x))c_{s, i+1}^{m_s}(C_{M_{s-1}, i+1}(w))=\\
&=&\Sigma_{M_{s},i}(x)C_{M_{s}, i+1}(w).
\end{eqnarray*}
To sum up, we conclude the recursion (\ref{11}), which finishes the proof of the lemma.
\end{proof}

Now, let us assume that the product $C_{M, i}=c_{1, i}^{m_1}\cdot\ldots \cdot c_{n, i}^{m_n}$ represents the identity mapping for some $M=(m_1,\ldots, m_n)\in~\mathcal{V}$.
Let $n_1<s\leq n$. Since there is $x_0\in X_i$ which belongs to the cycle $\sigma_{s, i}$ and $x_0\notin\{x_{1, i},\ldots, x_{n, i}\}$, we have by Lemma~\ref{lem1}: $x_0=C_{M,i}(x_0)=\Sigma_{M, i}(x_0)=\sigma_{s, i}^{m_s}(x_0)$. Thus the divisibility $r_{s-n_1}\mid m_s$ holds and consequently $m_s=0$. If $1\leq s\leq n_1$, then there is $i_0>i$ such that the length of the cycle $\sigma_{s, i_0}$ is greater than $|m_s|$. Let $w=x_0x_1\ldots x_{i_0-i}\in X^*_{(i)}$ be a word such that the last letter $x_{i_0-i}$ belongs to the cycle $\sigma_{s, i_0}$ and $x_j\notin\{x_{1, j+i},\ldots, x_{n, j+i}\}$ for every $0\leq j\leq i-i_0$. By Lemma~\ref{lem1} we obtain
$$
C_{M, i}(w)=\Sigma_{M, i}(x_0)\Sigma_{M, i+1}(x_1)\ldots \Sigma_{M, i_0}(x_{i_0-i}).
$$
Since $C_{M, i}(w)=w$, we have $x_{i_0-i}=\Sigma_{M, i_0}(x_{i_0-i})=\sigma_{s, i_0}^{m_s}(x_{i_0-i})$. Consequently $m_s=0$, which finishes the proof of Proposition~\ref{prop1}.
\end{proof}

\subsection{The groups $H_i$}
We consider the transformations
$$
d_{k, s, i}=a_i^{-k}\cdot c_{s, i}\cdot a_i^k,\;\;\;i\geq 0,\;\;\;k\in\mathbb{Z},\;\;\;1\leq s\leq n.
$$
For every  $i\geq 0$ we define the group
$$
H_i=\langle d_{k, s, i}\colon k\in\mathbb{Z},\; 1\leq s\leq n\rangle.
$$
For  every $i\geq 0$ and $k\in\mathbb{Z}$ we also define the  subgroup $K^{(k)}_i\leq H_i$, where
$$
K^{(k)}_i=a_i^{-k}K_ia_i^k=\{a_i^{-k}C_{M,i}a_i^k\colon M\in\mathcal{V}\}.
$$
By Proposition~\ref{prop1} we have the isomorphism $K^{(k)}_i\simeq K_i\simeq K$.

\begin{prop}\label{prop2}
The group $H_i$ is a direct sum of its subgroups $K_i^{(k)}$ for $k\in\mathbb{Z}$ and the mapping
$$
d_{k,s,i}\mapsto (\ldots,{\bf 0},{\bf 0},\kappa_s,{\bf 0},{\bf 0},\ldots),\;\;\;k\in\mathbb{Z},\;\;\;1\leq s\leq n,
$$
where the generator $\kappa_s\in K$ on the right hand side is in the $k$-th position, induces the isomorphism $H_i\simeq \bigoplus_\mathbb{Z}K$.
\end{prop}
\begin{proof}
By the recursion for $a_i$ and by~(\ref{5}) we easily obtain:
$$
d_{k, s, i}(xw)=\left\{\begin{array}{ll}
xw, &x=\alpha_i^k(x_{s, i}),\\
\sigma_{s, i}(x)d_{k, s, i+1}(w), &x\neq \alpha_i^k(x_{s, i})
\end{array}
\right.
$$
for all $x\in X_i$ and $w\in X^*_{(i+1)}$. Further, by the recursion for $a_i$ and by~(\ref{14}) we obtain for all $k, k'\in\mathbb{Z}$, $1\leq s, s'\leq n$, $x\in X_i$ and $w\in X^*_{(i+1)}$ the following recursion:
$$
d_{k, s, i}d_{k', s',i}(xw)=\left\{\begin{array}{ll}
xd_{k', s',i+1}(w),&x=\alpha_i^k(x_{s,i}),\\
xd_{k, s , i+1}(w), &x=\alpha_i^{k'}(x_{s',i}),\\
\sigma_{s,i}\sigma_{s',i}(x)d_{k, s,i+1}d_{k', s',i+1}(w), &\mbox{\rm otherwise},
\end{array}\right.
$$
and hence
\begin{eqnarray*}
d_{k,s,i}d_{k', s',i}(xw)=\left\{\begin{array}{ll}
d_{k', s', i}d_{k, s,i}(xw),&x\in\{\alpha_i^k(x_{s,i}),\alpha_i^{k'}(x_{s',i})\},\\
\sigma_{s,i}\sigma_{s',i}(x)d_{k, s,i+1}d_{k', s',i+1}(w), &\mbox{\rm otherwise}.
\end{array}\right.
\end{eqnarray*}
Consequently, the group $H_i$ is abelian. For $k\in\mathbb{Z}$, $M\in\mathcal{V}$ and $i\geq 0$ we  denote
$$
d_{k, M, i}=a_i^{-k}C_{M,i}a_i^k.
$$
In particular, we have $K_i^{(k)}=\{d_{k, M, i}\colon M\in\mathcal{V}\}$.
\begin{lem}\label{lem3}
If $M=(m_1, \ldots, m_n)\in\mathcal{V}$, then
$$
d_{k, M, i}=d_{k, 1, i}^{m_1}\cdot\ldots\cdot d_{k, n, i}^{m_n}.
$$
\end{lem}
\begin{proof}[of Lemma~\ref{lem3}]
We have
\begin{eqnarray*}
d_{k, M,i}&=&a_i^{-k}\cdot c_{1, i}^{m_1}\cdot\ldots\cdot c_{n, i}^{m_n}\cdot a_i^k=\\
&=&(a_i^{-k}\cdot c_{1, i}^{m_1}\cdot a_i^k)\cdot (a_i^{-k}\cdot c_{2, i}^{m_2}\cdot a_i^k)\cdot\ldots\cdot (a_i^{-k}\cdot c_{n, i}^{m_n}\cdot a_i^k)=\\
&=&(a_i^{-k}\cdot c_{1, i}\cdot a_i^k)^{m_1}\cdot\ldots\cdot (a_i^{-k}\cdot c_{n, i}\cdot a_i^k)^{m_n}=\\
&=&d_{k, 1, i}^{m_1}\cdot\ldots\cdot d_{k, n, i}^{m_n},
\end{eqnarray*}
which finishes the proof of Lemma~\ref{lem3}.
\end{proof}

Since the group $H_i$ is abelian, we see by Lemma~\ref{lem3} that every element of $H_i$ can be presented as a product
\begin{equation}\label{18}
d_{k_1, M_1, i}\cdot\ldots\cdot d_{k_t, M_t, i},\;\;\;t\geq 1,
\end{equation}
where
$$
k_1<k_2<\ldots<k_t,\;\;\;k_j\in\mathbb{Z},\;\;\; M_j\in \mathcal{V},\;\;\;1\leq j\leq t.
$$
We call the weight of the product (\ref{18}) as a number
$$
||M_1||+\ldots+||M_t||,
$$
where $||M||=|m_1|+\ldots+|m_n|$ for any $M=(m_1, \ldots, m_n)\in\mathcal{V}$. Now, we have to show that every element of $H_i$ is represented uniquely by such a product. Suppose not. Then there are products (\ref{18}) with nonzero weights defining the identity transformation. By Proposition~\ref{prop1} it must be $t\geq 2$ in every such product.  Let us choose $t\geq 2$, $i_0\geq 0$ and  the triples
$$
(k_1, M_1, i_0), \ldots, (k_t, M_t, i_0)\in \mathbb{Z}\times \mathcal{V}\times \{i_0\},
$$
such that the  product
$$
d_{k_1, M_1, i_0}\cdot\ldots\cdot d_{k_t, M_t, i_0}
$$
defines the identity transformation and its nonzero weight is minimal.  We can  assume  that each $M_j\in\mathcal{V}$ ($1\leq j\leq t$) is a nonzero vector.  For every $i\geq 0$ we define the product
$$
D_i=d_{k_1, M_1, i}\cdot\ldots\cdot d_{k_t, M_t, i}.
$$
There is a letter $z_i\in X_i$ ($i\geq 0$) which does not belong to any of the cycles $\pi_{s, i}$ ($1\leq s\leq n$). In particular $z_i\neq\alpha_i^{k_1}(x_{s, i})$ for every $1\leq s\leq n$. For every $1\leq s\leq n$ we also have
\begin{equation}\label{eqqq}
\Sigma_{M_1, i}\cdot \Sigma_{M_2, i}\cdot\ldots\cdot\Sigma_{M_j, i}(z_i)\neq \alpha_i^{k_{j+1}}(x_{s, i}),\;\;\;1\leq j\leq t.
\end{equation}
Let us  denote:
$$
z_i'=\Sigma_{M_1, i}\cdot\ldots\cdot\Sigma_{M_t, i}(z_i),\;\;\;i\geq0.
$$
For all $i\geq 0$, $k\in\mathbb{Z}$, $M\in\mathcal{V}$, $x\in X_i$ and $w\in X^*_{(i+1)}$ we obtain by Lemma~\ref{lem1}  the following recursion
\begin{equation}\label{lem2}
d_{k, M, i}(xw)=\left\{
\begin{array}{ll}
xd_{k, M^{(s)}, i+1}(w), &x=\alpha^k_i(x_{s,i}),\\
\Sigma_{M, i}(x)d_{k, M, i+1}(w), &\mbox{\rm otherwise}.
\end{array}
\right.
\end{equation}
The above recursion and the inequalities~(\ref{eqqq}) imply
\begin{equation}\label{16}
D_i(z_iz_{i+1}\ldots z_{i+j}w)=z'_iz'_{i+1}\ldots z'_{i+j}D_{i+j+1}(w)
\end{equation}
for all $i, j\geq 0$ and  $w\in X^*_{(i+j+1)}$. Since $D_{i_0}$ defines the identity mapping, we obtain by (\ref{16}) that for every $i\geq i_0$ the transformation $D_{i}$ also defines the identity mapping. Let $1\leq s\leq n$ be such that the $s$-th component of $M_1$ is not equal to 0. In particular $||M_1^{(s)}||<||M_1||$. Let  $\nu\geq i_0$ be such that the length $L$ of the cycle $\pi_{s, \nu}$ is greater than the sum $|k_1|+\ldots+|k_t|$.
Let us consider the letter
$$
z=\alpha_{\nu}^{k_1}(x_{s, \nu})=\pi_{s, \nu}^{k_1}(x_{s, \nu})\in X_{\nu}.
$$
By the equality~(\ref{lem2}) we have
\begin{equation}\label{22}
d_{k_1, M_1, \nu}(zw)=zd_{k_1, M_1^{(s)}, \nu+1}(w)
\end{equation}
for every $w\in X^*_{(\nu+1)}$.
For every $2\leq j\leq t$ we obviously have:
\begin{equation}\label{20}
\Sigma_{M_2, \nu}\cdot\ldots\cdot \Sigma_{M_j, \nu}(z)=z.
\end{equation}
For $2\leq j\leq t$ we also have
\begin{equation}\label{21}
z\notin\{\alpha_\nu^{k_j}(x_{1, \nu}),\ldots, \alpha_\nu^{k_j}(x_{n, \nu})\}.
\end{equation}
Indeed, suppose to othe contrary that
$$
z=\alpha_\nu^{k_j}(x_{s', \nu})=\pi_{s', \nu}^{k_j}(x_{s', \nu})
$$
for some $2\leq j\leq t$ and $1\leq s'\leq n$. Then we have
$$
\pi_{s, \nu}^{k_1}(x_{s, \nu})=\pi_{s', \nu}^{k_j}(x_{s', \nu}).
$$
Since the cycles $\pi_{1, \nu}, \ldots, \pi_{n, \nu}$ are pairwise disjoint, we obtain $s=s'$. Hence  $\pi_{s,\nu}^{k_j-k_1}(x_{s, \nu})=x_{s, \nu}$ and  consequently $L\mid k_j-k_1$, which contradicts with the inequality $L>|k_1|+\ldots+|k_t|$. By (\ref{lem2}) and by (\ref{22})-(\ref{21}) we obtain $D_\nu(zw)=z D(w)$ for all $w\in X^*_{(\nu+1)}$, where
$$
D=d_{k_1, M^{(s)}_1, \nu+1}\cdot d_{k_2, M_2, \nu+1}\cdot\ldots\cdot d_{k_t, M_t, \nu+1}.
$$
Since $D_\nu$ defines the identity mapping, the product $D$ also defines the identity mapping. Since  $||M_1^{(s)}||<||M_1||$ and each $M_j$ ($2\leq j\leq t$) is a nonzero vector, we see that  $D$ has a nonzero weight  which is smaller than the weight of $D_{i_0}$,  contrary to our assumption. This finishes the proof of Proposition~\ref{prop2}.
\end{proof}

\begin{lem}\label{lem4} We have
$$
a_i^{-k}\cdot d_{k_1, M_1, i}\cdot\ldots\cdot d_{k_t, M_t, i}\cdot a_i^k=d_{k+k_1, M_1, i}\cdot\ldots\cdot d_{k+k_t, M_t, i}
$$
for all $i\geq 0$, $t\geq 1$, $k, k_j\in\mathbb{Z}$ and $M_j\in\mathcal{V}$ ($1\leq j\leq t$).

\end{lem}
\begin{proof} The left side of the above equality is equal to
$$
(a_i^{-k}\cdot d_{k_1, M_1, i}\cdot a_i^k)\cdot (a_i^{-k}\cdot d_{k_2, M_2, i}\cdot a_i^k)\cdot\ldots\cdot (a_i^{-k}\cdot d_{k_t, M_t, i}\cdot a_i^k).
$$
Now, it suffices to see that the $j$-th factor ($1\leq j\leq t$) of the above product is equal to
\begin{eqnarray*}
a_i^{-k}\cdot d_{k_j, M_j, i}\cdot a_i^k&=&a_i^{-k}\cdot (a_i^{-k_j}\cdot C_{M_j, i}\cdot a_i^{k_j})\cdot a_i^k=\\
&=&a_i^{-(k+k_j)}\cdot C_{M_j, i}\cdot a_i^{k+k_j}=d_{k+k_j, M_j, i}.
\end{eqnarray*}
This finishes the proof of Lemma~\ref{lem4}.
\end{proof}

Let us fix $1\leq s\leq n$. For all $i, j\geq 0$ let us consider the words $w_{i, j}\in X^*_{(i)}$, where
$$
w_{i, j}=x_{s, i}x_{s, i+1}\ldots x_{s, i+j}.
$$
\begin{lem}\label{lem5} For every $i\geq 0$ the following hold:
\begin{itemize}
\item[(i)] the group $H_i$ is contained in the stabilizer of any $w_{i, j}$,
\item[(ii)] for every nonzero integer $k$ there is $j\geq0$ such that $a_i^k(w_{i,j})\neq w_{i, j}$.
\end{itemize}
\end{lem}
\begin{proof}
For $j\geq 0$ and $M\in\mathcal{V}$ we have $\Sigma_{M, i+j}(x_{s, i+j})=x_{s, i+j}$. Thus by the equality~(\ref{lem2}) we see that for every $j\geq 0$ the transformations $d_{k, M, i+j}$ ($k\in\mathbb{Z}$,  $M\in\mathcal{V}$)  stabilize the letter $x_{s, i+j}$. Consequently,  the elements $d_{k, M, i}$ ($k\in\mathbb{Z}$, $M\in\mathcal{V}$)  stabilize the words $w_{i, j}$ ($j\geq 0$). Since the elements generate the group $H_i$, we obtain~(i). To show (ii) we see by the recursion for $a_i$ that for every $j\geq0$ we have:
$$
a_i^k(w_{i,j})=\alpha_i^k(x_{s,i})\ldots\alpha_{i+j}^k(x_{s, i+j})=\pi_{s,i}^k(x_{s,i})\ldots\pi_{s, {i+j}}^k(x_{s, i+j}).
$$
Now, the claim follows from the assumption that the lengths of  cycles $\pi_{s, i+j}$ ($j\geq 0$) are not uniformly bounded.
\end{proof}

\subsection{The groups $G_i$}
Let us define the groups
$$
G_i=\langle a_i, c_{1, i},\ldots, c_{n, i}\rangle=\langle K_i, a_i\rangle,\;\;\;i\geq 0.
$$
Directly by the definition of the group $H_i$ we obtain $K_i\leq H_i\triangleleft G_i$. By Lemma~\ref{lem5} the group $H_i$ intersects trivially with the cyclic group $\langle a_i\rangle$. Thus $G_i$ is a semidirect product of  the subgroups $H_i $ and $\langle a_i\rangle$. By  Proposition~\ref{prop2} the subgroup $H_i$ is isomorphic with the direct sum $\bigoplus_\mathbb{Z}K$ via
$$
d_{k,s,i}\mapsto (\ldots,{\bf 0},{\bf 0},\kappa_s,{\bf 0},{\bf 0},\ldots),\;\;\;k\in\mathbb{Z},\;\;\;1\leq s\leq n,
$$
where $\kappa_s\in K$ on the right hand side is in the $k$-th position. By Lemma~\ref{lem5}  the cyclic group $\langle a_i\rangle$ is isomorphic, via $a_i\mapsto u$, with $\mathbb{Z}$. By Lemma~\ref{lem4} this cyclic group acts on $H_i$ by conjugation via the shift. Consequently, we obtain the isomorphism $G_i\simeq \bigoplus_\mathbb{Z}K\rtimes \mathbb{Z}= K\wr \mathbb{Z}$, which is induced by the mapping $a_i\mapsto u$, $c_{s,i}\mapsto \eta_s$ ($1\leq s\leq n$).
Further, for the group $G(\mathcal{A}_{i})$ generated by the $i$-th shift ($i\geq0$) of the automaton $\mathcal{A}$ we have
$$
G(\mathcal{A}_{i})=\langle a_i, (b_1)_i,\ldots, (b_n)_i\rangle.
$$
Since $c_{s,i}=(b_s)_ia_i^{-1}$, we obtain $G(\mathcal{A}_{i})=G_i$ and  the mapping
$$
a_i\mapsto u,\;\;\;(b_s)_i\mapsto \eta_s\cdot u,\;\;\;1\leq s\leq n
$$
induces an isomorphism $G(\mathcal{A}_{i})\simeq K\wr\mathbb{Z}$, which finishes the proof of Theorem~\ref{t2}.

\section{Examples}
\subsection{The lamplighter group $\mathbb{Z}\wr\mathbb{Z}$}
The lamplighter group $\mathbb{Z}\wr\mathbb{Z}$ is an example of a 2-generated torsion-free  not finitely presented metabelian group. It is generated by two elements $u$, $\eta$, where $u$ is a generator of $\mathbb{Z}$ in the semidirect product $\bigoplus_\mathbb{Z}\mathbb{Z}\rtimes \mathbb{Z}$ and
$$
\eta=(\ldots,0,0,0,u,0,0,0,\ldots),
$$
where   $u$ on the right hand side is in zero position. This group has the following presentation $\langle u, \eta\colon [\eta^{u^k},\eta^{u^{k'}}]=1,\;\;k,k'\in\mathbb{Z}\rangle$,
where $[x,y]=x^{-1}y^{-1}xy$.

Basing on the  construction of the automaton $\mathcal{A}$ we can present the lamplighter group $\mathbb{Z}\wr \mathbb{Z}$ as a group defined by a 2-state self-similar automaton over the changing alphabet as follows.
Let $(r_i)_{i\geq 0}$ be an arbitrary unbounded sequence of integers $r_i>1$ and let $X=(X_i)_{i\geq 0}$ be the changing alphabet with
$$
X_i=\{1,2,\ldots, 2r_i\}.
$$
Let $A=(X, Q, \varphi, \psi)$ be an automaton with the 2-element set $Q=\{a, b\}$ of internal states and the following sequences $\varphi=(\varphi_i)_{i\geq 0}$, $\psi=(\psi_i)_{i\geq 0}$ of transition and output functions:
$$
\varphi_i(q, x)=\left\{\begin{array}{ll}
a, &q=b,\;\;x=1,\\
q, &\mbox{\rm otherwise},
\end{array}\right.\;\;\;
\psi_i(q,x)=\left\{\begin{array}{ll}
\alpha_i(x),&q=a,\\
\beta_{i}(x), &q=b,
\end{array}\right.
$$
where the permutations $\alpha_i$, $\beta_i$ of the set $X_i$ may be defined, for example, as follows:
$$
\alpha_i=(1,3,\ldots, 2r_i-1),\;\;\;\beta_i=(1,3,\ldots, 2r_i-1)(2,4,\ldots,2r_i).
$$
By above, we see that the automaton $A$ is self-similar and the group $G(A)$ defined by $A$ is isomorphic with the lamplighter group $\mathbb{Z}\wr\mathbb{Z}$ via  $a_0\mapsto u$, $c_0\mapsto \eta$, where $c_0=b_0a_0^{-1}$.

\subsection{The lamplighter groups $\mathbb{Z}^n\wr \mathbb{Z}$, $n\geq 1$}

Generalizing the above construction, we can describe for every positive integer $n$ the wreath product $\mathbb{Z}^n\wr \mathbb{Z}$ as a group defined by the following  self-similar automaton $A=(X, Q, \varphi, \psi)$:
\begin{itemize}
\item $X=(X_i)_{i\geq 0}$, where $X_i=\{1,2,\ldots, 2nr_i\}$ and $(r_i)_{i\geq 0}$ is an arbitrary unbounded sequence of integers greater than 1,
\item $Q=\{a, b_0, \ldots, b_{n-1}\}$,
\item $
\varphi_i(q, x)=\left\{\begin{array}{ll}
a, &q=b_s,\;\;x=2sr_i+1,\;\;0\leq s\leq n-1,\\
q, &\mbox{\rm otherwise},
\end{array}\right.
$
\item $
\psi_i(q,x)=\left\{\begin{array}{ll}
\alpha_i(x),&q=a,\\
\beta_{s, i}(x), &q=b_s,\;\;0\leq s\leq n-1,
\end{array}\right.
$
\end{itemize}
where the permutations $\alpha_i$, $\beta_{s, i}$ of the set $X_i$ are defined as follows
\begin{eqnarray*}
\alpha_i&=&\prod\limits_{s=0}^{n-1}(2sr_i+1, 2sr_i+3,\ldots, 2sr_i+2r_i-1),\\
\beta_{s, i}&=&\alpha_i\cdot (2sr_i+2, 2sr_i+4, \ldots, 2sr_i+2r_i).
\end{eqnarray*}

\subsection{The lamplighter groups $C_r\wr\mathbb{Z}$, $r>1$}
For  lamplighter groups of the form $C_r\wr\mathbb{Z}$ ($r>1$) the corresponding self-similar automaton representation $A=(X, Q, \varphi, \psi)$ may be described as follows:
\begin{itemize}
\item $X_i=\{1,2,\ldots, r_i, r_i+1, r_i+2,\ldots, r_i+r\}$, where $(r_i)_{i\geq 0}$ is an arbitrary unbounded sequence of  integers greater than 1,
\item $Q=\{a, b\}$,
\item $
\varphi_i(q, x)=\left\{\begin{array}{ll}
a, &q=b,\;\;x=1,\\
q, &\mbox{\rm otherwise},
\end{array}\right.
$
\item $\psi_i(q,x)=\left\{\begin{array}{ll}
\alpha_i(x),&q=a,\\
\beta_{i}(x), &q=b,
\end{array}\right.
$
\end{itemize}
where the permutations $\alpha_i,\beta_i\in Sym(X_i)$ may be defined, for example, as follows:
$$
\alpha_i=(1,2,\ldots, r_i),\;\;\;\beta_i=(1,2,\ldots, r_i)(r_i+1, r_i+2, \ldots, r_i+r).
$$

\end{document}